\documentclass[11pt,a4paper]{article}

\usepackage[T1]{fontenc}
\usepackage[utf8]{inputenc}
\usepackage[english]{babel}
\usepackage{lmodern}
\usepackage{microtype}
\usepackage{geometry}
\usepackage{amsmath,amssymb,amsthm,mathtools}
\usepackage{graphicx}
\usepackage{booktabs,longtable,array}
\usepackage{enumitem}
\usepackage[numbers,sort&compress]{natbib}
\usepackage{xcolor}
\usepackage{hyperref}
\hypersetup{colorlinks=true,linkcolor=blue!60!black,citecolor=blue!60!black,urlcolor=blue!60!black}
\usepackage[nameinlink,noabbrev]{cleveref}

\graphicspath{{figures/}}
\setlist[itemize]{leftmargin=2em}
\setlist[enumerate]{leftmargin=2em}
\newcommand{\ZZ}{\mathbb Z}

\newcommand{\Ber}{\mathcal B}
\newcommand{\diag}{\operatorname{diag}}
\newcommand{\lcm}{\operatorname{lcm}}

\newcommand{\tr}{\operatorname{tr}}

\theoremstyle{definition}
\newtheorem{definition}{Definition}[section]
\newtheorem{remark}[definition]{Remark}

\theoremstyle{plain}
\newtheorem{proposition}[definition]{Proposition}
\newtheorem{theorem}[definition]{Theorem}
\newtheorem{corollary}[definition]{Corollary}

\title{Star-Shaped Integral Cartan-Type Matrices and an Egyptian-Fraction Classification of Affine Weighted Trees}
\author{E. Torrente-Lujan\\
\small Grupo de F\'isica Te\'orica, Departamento de F\'isica, Universidad de Murcia, Spain\\
\small \texttt{e.torrente@cern.ch}}
\date{Matrix/graph oriented revision, \today}

\begin{document}
\maketitle

\begin{abstract}
We study a concrete family of symmetric integral $Z$-matrices attached to weighted star trees.  The arms are ordinary type-$A$ chains and the central diagonal entry is an arbitrary positive integer $k$ rather than being fixed to the Cartan value $2$.  This gives a matrix-theoretic and graph-theoretic version of the so called Berger construction: it extends the simply laced affine Dynkin stars while remaining accessible through elementary linear algebra.  For a star with arm lengths $r_1,\ldots,r_m$ we compute the determinant, the inertia, the positive-definite and affine regimes, and the primitive positive null vector in the affine case.  The affine condition is exactly the unit-fraction equation
\[
   \sum_{i=1}^m \frac{1}{r_i+1}=m-k,
\]
so the classification of these affine weighted trees reduces to a finite Egyptian-fraction enumeration for each fixed pair $(m,k)$.  The classical affine diagrams $D_4^{(1)}$, $E_6^{(1)}$, $E_7^{(1)}$ and $E_8^{(1)}$ appear as small subfamilies, while higher-arm cases give new integral positive-semidefinite star matrices with explicit Coxeter labels.  
\end{abstract}


\section{Introduction}

Dynkin diagrams sit at a meeting point of linear algebra, graph theory, Lie theory and geometry.  In the simply laced case a diagram is a graph $G$ and the associated Cartan matrix is $2I-A(G)$, where $A(G)$ is the adjacency matrix.  Positivity and semipositivity of this matrix are equivalent to the familiar finite and affine ADE classifications.  From the graph-theoretic side this is closely related to Smith's classification of connected graphs with adjacency spectral radius at most $2$, later developed in the spectral graph theory literature \citep{Smith1970,Hoffman1972,HoffmanSmith1975,Cvetkovic1980,BrouwerHaemers2012,GodsilRoyle2001}.  From the matrix side, the same theme appears in combinatorial matrix theory and in the classification of integer symmetric, signed and symmetrizable matrices whose eigenvalues lie in prescribed intervals \citep{Bapat2010,BrualdiCvetkovic2008,McKeeSmyth2007,McKeeSmyth2020}.

The purpose of this paper is to place the so called Berger-type diagrams into this matrix/graph framework.  There is also a nearby representation-theoretic tradition in which graphs, quadratic forms and Cartan-type matrices control finite or tame classification problems \citep{Gabriel1972,Kac1980,Ringel1984,GeigleLenzing1987,FominZelevinsky2003,BarotGeissZelevinsky2006}.  Rather than treating them first as hypothetical new Lie-theoretic objects, we regard them as explicit integral symmetric $Z$-matrices: the off-diagonal entries are non-positive, the underlying graph is a tree, and some diagonal entries are allowed to differ from $2$.  The family considered here is the most tractable one: a weighted star tree whose arms are ordinary type-$A$ chains and whose central vertex has diagonal weight $k\in\ZZ_{>0}$.  Thus the associated matrix is a weighted analogue of $2I-A(G)$, with only one non-standard diagonal entry.  The positive and null labellings appearing below are also close to Vinberg's additive and subadditive functions on graphs, their use in the representation theory of Euclidean diagrams, and recent classifications of quivers with additive or subadditive labellings \citep{Vinberg1971,HappelPreiserRingel1980,GalashinPylyavskyy2019,GalashinPylyavskyy2020}.

This viewpoint has two advantages.  First, it separates rigorous linear-algebraic statements from broader geometric motivation.  The determinant, inertia, affine condition and null vector are consequences of Schur complements and elementary properties of the $A_r$ Cartan matrices.  Secondly, it turns the affine classification into a precise combinatorial problem: after setting $N_i=r_i+1$, affine matrices are exactly the solutions of
\begin{equation}
        \frac1{N_1}+\cdots+\frac1{N_m}=m-k,
        \qquad N_i\ge 2.
\label{eq:intro-egyptian}
\end{equation}
Consequently the enumeration of affine weighted stars is an Egyptian-fraction enumeration with a graph-theoretic interpretation.

There are also close representation-theoretic analogues of this matrix/graph viewpoint.  Gabriel's theorem and Kac's extension relate graph supports, quadratic forms and indecomposable quiver representations to Dynkin and Kac--Moody root systems \citep{Gabriel1972,Kac1980}.  Weighted projective lines and canonical algebras use star-shaped diagrams whose weight sequences are governed by expressions such as $2-m+\sum_i 1/p_i$, separating domestic, tubular and wild regimes \citep{Ringel1984,GeigleLenzing1987}.  Cluster finite-type classifications again connect matrices, graphs and Cartan--Killing types through mutation and positive quasi-Cartan data \citep{FominZelevinsky2003,BarotGeissZelevinsky2006}.  These works are not prerequisites for the proofs below, but they identify nearby mathematical literatures for the present weighted-star classification.

The results below may be read as a controlled extension of the finite/affine Dynkin story.  The ordinary affine stars are recovered in small cases: $D_4^{(1)}$ arises from four arms of length one and central weight $2$, while $E_6^{(1)}$, $E_7^{(1)}$ and $E_8^{(1)}$ arise from three arms and central weight $2$.  The larger families do not automatically define new Kac--Moody algebras; they are classified positive-semidefinite integral star matrices with prescribed sign pattern and a primitive positive kernel vector.

The paper is organized as follows.  \Cref{sec:framework} fixes the matrix and graph conventions.  \Cref{sec:star} defines the star matrices.  \Cref{sec:det-inertia} proves the determinant and inertia formulae and extracts the affine condition.  \Cref{sec:classification} gives the Diophantine classification algorithm.  \Cref{sec:labels} computes the Coxeter labels.  \Cref{sec:enumeration} lists the first families, including the classical affine diagrams and higher-arm Berger-type weighted stars.

\section{Weighted star matrices and graph convention}
\label{sec:framework}

A real matrix with non-positive off-diagonal entries is often called a $Z$-matrix; positive-definite or singular positive-semidefinite examples belong to the general circle of Stieltjes and $M$-matrix theory \citep{BermanPlemmons1994,HornJohnson2013}.  The matrices in this paper are integral symmetric $Z$-matrices with a prescribed tree support.

\begin{definition}[Cartan-type weighted graph matrix]
Let $G=(V,E)$ be a finite simple graph and let $w:V\to\ZZ_{>0}$ be a vertex-weight function.  The associated symmetric matrix is
\begin{equation}
        B(G,w)_{uv}=\begin{cases}
        w(u), & u=v,\\
        -1, & \{u,v\}\in E,\\
        0, & \text{otherwise.}
        \end{cases}
\label{eq:weighted-graph-matrix}
\end{equation}
If $B(G,w)$ is positive definite, we call the weighted graph finite type.  If $B(G,w)$ is positive semidefinite of corank one and every proper principal submatrix is positive definite, we call it affine type.
\end{definition}

For $w\equiv 2$, this is the usual simply laced generalized Cartan matrix.  The present paper allows $w$ to differ from $2$ at the central vertex of a star.  The term \emph{Berger matrix} is retained as historical terminology for such weighted Cartan-type matrices, but the main results are stated entirely in terms of \eqref{eq:weighted-graph-matrix}.

If $B$ is affine type, Perron--Frobenius theory applied to the irreducible non-negative matrix obtained by shifting $B$ by a large scalar multiple of the identity gives a one-dimensional kernel generated by a strictly positive vector.  We call the primitive integral generator the vector of \emph{Coxeter labels}; its coordinate sum is the corresponding \emph{Coxeter number}.  In graph language this is an additive labelling of the weighted tree.  This terminology matches the usual affine Dynkin convention, but no representation-theoretic construction is assumed.

\section{The star-shaped family}
\label{sec:star}

Let $A_r$ denote the Cartan matrix of type $A_r$,
\[
A_r=\begin{pmatrix}
2&-1&0&\cdots&0\\
-1&2&-1&\ddots&\vdots\\
0&-1&2&\ddots&0\\
\vdots&\ddots&\ddots&\ddots&-1\\
0&\cdots&0&-1&2
\end{pmatrix}.
\]
Fix $m\geq 2$, positive integers $r_1,\ldots,r_m$, and a central weight $k\in\ZZ_{>0}$.  Let $v_i\in\ZZ^{r_i}$ be the column vector with all entries zero except the last entry, which is $-1$.  The star-shaped matrix studied in this paper is
\begin{equation}
B(k;r_1,\ldots,r_m)=
\begin{pmatrix}
A_{r_1} &0&\cdots&0&v_1\\
0&A_{r_2}&\cdots&0&v_2\\
\vdots&\vdots&\ddots&\vdots&\vdots\\
0&0&\cdots&A_{r_m}&v_m\\
v_1^t&v_2^t&\cdots&v_m^t&k
\end{pmatrix}.
\label{eq:star-matrix}
\end{equation}
The underlying graph is a tree with $m$ linear arms attached to one central vertex.  Its dimension is
\begin{equation}
D=1+\sum_{i=1}^m r_i,
\label{eq:dimension}
\end{equation}
and its trace is
\begin{equation}
\tr B=2\sum_{i=1}^m r_i+k.
\end{equation}
The sum of all matrix entries is
\begin{equation}
K^2(B):=\sum_{a,b}B_{ab}=k.
\label{eq:K-square}
\end{equation}
In the previous geometric motivation this number was interpreted as an intersection-theoretic invariant.  In the present matrix/graph treatment it is simply the total entry sum of the weighted star matrix.

\section{Determinant, inertia and affine condition}
\label{sec:det-inertia}

\begin{proposition}[Determinant]
For the matrix \eqref{eq:star-matrix},
\begin{equation}
\det B(k;r_1,\ldots,r_m)=
\prod_{i=1}^m (r_i+1)\left(k-m+\sum_{i=1}^m \frac{1}{r_i+1}\right).
\label{eq:determinant}
\end{equation}
\end{proposition}

\begin{proof}
Use the Schur complement with respect to the block diagonal matrix $\diag(A_{r_1},\ldots,A_{r_m})$.  The standard identities
\[
\det A_r=r+1,\qquad (A_r^{-1})_{rr}=\frac{r}{r+1}
\]
give
\begin{align*}
\det B
&=\prod_{i=1}^m \det A_{r_i}\left(k-\sum_{i=1}^m v_i^t A_{r_i}^{-1}v_i\right)\\
&=\prod_{i=1}^m (r_i+1)\left(k-\sum_{i=1}^m\frac{r_i}{r_i+1}\right)\\
&=\prod_{i=1}^m (r_i+1)\left(k-m+\sum_{i=1}^m\frac{1}{r_i+1}\right).
\end{align*}
\end{proof}

The determinant is only the first consequence of the same calculation.  The full signature is controlled by a scalar Schur complement.

\begin{theorem}[Inertia of the weighted star]
Let
\begin{equation}
        S(k;r_1,\ldots,r_m)=k-m+\sum_{i=1}^m\frac{1}{r_i+1}.
\label{eq:schur-scalar}
\end{equation}
Then the inertia of $B(k;r_1,\ldots,r_m)$ is the inertia of the positive-definite block $\diag(A_{r_1},\ldots,A_{r_m})$ plus the inertia of the scalar $S$.  Hence:
\begin{enumerate}[label=\textup{(\roman*)}]
\item $B$ is positive definite if and only if $S>0$;
\item $B$ is positive semidefinite of corank one if and only if $S=0$;
\item $B$ has exactly one negative eigenvalue if and only if $S<0$.
\end{enumerate}
Moreover, in the case $S=0$, every proper principal submatrix is positive definite.
\end{theorem}

\begin{proof}
The Haynsworth inertia additivity formula for Schur complements gives the first assertion because each $A_{r_i}$ is positive definite.  The three cases follow immediately.  For the final statement, if the central vertex is deleted then the remaining matrix is a direct sum of positive-definite path Cartan matrices.  If the central vertex remains but a vertex is deleted from an arm, then the part of that arm still attached to the central vertex has length $\ell_i\le r_i$, with at least one strict inequality.  The corresponding Schur contribution is $\ell_i/(\ell_i+1)<r_i/(r_i+1)$ for the shortened arm, so the scalar Schur complement becomes strictly positive when the original value was zero.
\end{proof}

\begin{corollary}[Affine condition]
The matrix \eqref{eq:star-matrix} is affine type precisely when
\begin{equation}
\sum_{i=1}^m \frac{1}{r_i+1}=m-k.
\label{eq:affine-condition}
\end{equation}
In particular, if $p=m-k$, then affine weighted stars are equivalent to unit-fraction representations of the integer $p$ by $m$ denominators $r_i+1\ge2$.
\end{corollary}

For $k\geq m$, the determinant is positive and the matrix is positive definite.  Affine examples require $p=m-k>0$.  Since $r_i\geq1$, each summand in \eqref{eq:affine-condition} is at most $1/2$, and therefore $p\le m/2$.  The main family in the original manuscript is $p=1$:
\begin{equation}
\sum_{i=1}^m \frac{1}{r_i+1}=1.
\label{eq:egyptian-one}
\end{equation}
The next family is $p=2$:
\begin{equation}
\sum_{i=1}^m \frac{1}{r_i+1}=2.
\label{eq:egyptian-two}
\end{equation}

\section{Combinatorial classification by unit fractions}
\label{sec:classification}

Set
\begin{equation}
        N_i=r_i+1\qquad (N_i\ge2), \qquad p=m-k.
\end{equation}
After sorting the denominators, the affine classification problem becomes
\begin{equation}
        2\le N_1\le\cdots\le N_m, \qquad
        \sum_{i=1}^m \frac1{N_i}=p.
\label{eq:unit-fraction-classification}
\end{equation}
For fixed $(m,p)$ this is a finite search.  A standard recursive enumeration uses exact rational arithmetic.  Suppose that after choosing $N_1,\ldots,N_{j-1}$ the residual is $\rho>0$ and $t=m-j+1$ terms remain.  The next denominator $N_j$ must satisfy
\begin{equation}
        \left\lceil\frac1\rho\right\rceil \le N_j \le
        \left\lfloor\frac{t}{\rho}\right\rfloor,
\label{eq:egyptian-bounds}
\end{equation}
with the additional monotonicity condition $N_j\ge N_{j-1}$.  This gives a complete depth-first enumeration of \eqref{eq:unit-fraction-classification}; it is the same elementary search principle used in classical and algorithmic work on finite sums of unit fractions \citep{Kellogg1921,Curtiss1922,Graham1964,Eppstein1995,Sandor2003,Bloom2022}.

The table below records the exact number of unordered solutions for the small cases used later.
\begin{center}
\begin{tabular}{@{}r r r r@{}}
\toprule
$m$ & $p$ & $k=m-p$ & number of unordered affine stars\\
\midrule
2 & 1 & 1 & 1\\
3 & 1 & 2 & 3\\
4 & 1 & 3 & 14\\
5 & 1 & 4 & 147\\
6 & 1 & 5 & 3462\\
4 & 2 & 2 & 1\\
5 & 2 & 3 & 3\\
6 & 2 & 4 & 17\\
\bottomrule
\end{tabular}
\end{center}

\begin{remark}[Relation with spectral graph theory]
When all diagonal weights are $2$, positivity of $2I-A(G)$ is equivalent to $\rho(A(G))<2$ and affine semipositivity is equivalent to $\rho(A(G))=2$.  The classical simply laced Dynkin and affine Dynkin diagrams are exactly the connected graphs in this spectral regime.  

Small spectral-radius classification problems also have higher-dimensional and hypergraph variants \citep{LuMan2016}.  The present family changes the diagonal at one vertex, so it is not a direct instance of Smith's theorem; rather, it is a one-parameter weighted analogue for star trees, with the scalar Schur complement replacing the global spectral-radius test.
\end{remark}

\section{Coxeter labels}
\label{sec:labels}

\begin{proposition}[Coxeter labels]
Assume \eqref{eq:affine-condition}.  Put
\begin{equation}
N_i=r_i+1,
\qquad s=\lcm(N_1,\ldots,N_m).
\end{equation}
Then a primitive positive kernel vector of $B$ is given, up to a common divisor, by the central label $s$ and by the arm labels
\begin{equation}
 c_{i,j}=\frac{s}{r_i+1}\,j,
 \qquad j=1,\ldots,r_i.
\label{eq:coxeter-labels}
\end{equation}
The corresponding Coxeter number is
\begin{equation}
 h=s+\sum_{i=1}^m\sum_{j=1}^{r_i}c_{i,j}
   =s\left(1+\frac12\sum_{i=1}^m r_i\right)
   =\frac{s}{2}(D+1).
\label{eq:coxeter-number}
\end{equation}
\end{proposition}

\begin{proof}
Writing a kernel vector in block form $c^t=(c_1,\ldots,c_m,s)$ gives
\[
A_{r_i}c_i+s v_i=0,
\qquad
\sum_{i=1}^m v_i^t c_i+k s=0.
\]
The identity
\[
A_r(1,2,\ldots,r)^t=(r+1)(0,\ldots,0,1)^t
\]
solves the first equation and gives \eqref{eq:coxeter-labels}.  The last equation is exactly \eqref{eq:affine-condition}.  Summing the labels gives \eqref{eq:coxeter-number}.
\end{proof}

It is useful to multiply \eqref{eq:affine-condition} by $s$.  With
\begin{equation}
        x_i=\frac{s}{r_i+1},
\end{equation}
the affine condition becomes
\begin{equation}
        \sum_{i=1}^m x_i=p s,
        \qquad p=m-k,
\label{eq:divisor-partition}
\end{equation}
and each $x_i$ divides $s$.  The tables below display each solution in the form $(x_1,\ldots,x_m)[s]$.

\section{Enumeration of the first affine families}
\label{sec:enumeration}

We write
\begin{equation}
        \Ber^{(p)}(r_1,\ldots,r_m),
        \qquad p=m-k,
\end{equation}
for the affine weighted star determined by \eqref{eq:affine-condition}.  In the ordinary simply laced affine case $k=2$, four classical stars occur in this notation:
\begin{align*}
D_4^{(1)} &\simeq \Ber^{(2)}(1,1,1,1),\\
E_6^{(1)} &\simeq \Ber^{(1)}(2,2,2),\\
E_7^{(1)} &\simeq \Ber^{(1)}(1,3,3),\\
E_8^{(1)} &\simeq \Ber^{(1)}(1,2,5).
\end{align*}
The remaining rows are weighted-star extensions in which the central diagonal value is not necessarily $2$.  

The counts are for unordered solutions, i.e. after sorting the denominators $r_i+1$.  The lists for $m=5$ and $m=6$ in the $p=1$ case are truncated by the dimension bound stated in the captions; the exact total counts are given in \Cref{sec:classification}.

\subsection{The \texorpdfstring{$p=1$}{p=1} families}

\begin{longtable}{@{}r l r r r l@{}}
\caption{$m=2$, $k=1$, $p=1$.}\label{tab:m2p1}\\
\toprule
No. & Type & $D$ & $s$ & $h$ & $(x_i)[s]$\\
\midrule
\endfirsthead
\toprule
No. & Type & $D$ & $s$ & $h$ & $(x_i)[s]$\\
\midrule
\endhead
1 & $\Ber^{(1)}(1,1)$ & 3 & 2 & 4 & $(1,1)[2]$ \\
\bottomrule
\end{longtable}

\begin{longtable}{@{}r l r r r l@{}}
\caption{$m=3$, $k=2$, $p=1$: the three exceptional affine ADE star cases.}\label{tab:m3p1}\\
\toprule
No. & Type & $D$ & $s$ & $h$ & $(x_i)[s]$\\
\midrule
\endfirsthead
\toprule
No. & Type & $D$ & $s$ & $h$ & $(x_i)[s]$\\
\midrule
\endhead
1 & $\Ber^{(1)}(2,2,2)$ & 7 & 3 & 12 & $(1,1,1)[3]$ \\
2 & $\Ber^{(1)}(1,3,3)$ & 8 & 4 & 18 & $(2,1,1)[4]$ \\
3 & $\Ber^{(1)}(1,2,5)$ & 9 & 6 & 30 & $(3,2,1)[6]$ \\
\bottomrule
\end{longtable}

\begin{longtable}{@{}r l r r r l@{}}
\caption{$m=4$, $k=3$, $p=1$: all $14$ unordered solutions.}\label{tab:m4p1}\\
\toprule
No. & Type & $D$ & $s$ & $h$ & $(x_i)[s]$\\
\midrule
\endfirsthead
\toprule
No. & Type & $D$ & $s$ & $h$ & $(x_i)[s]$\\
\midrule
\endhead
1 & $\Ber^{(1)}(3,3,3,3)$ & 13 & 4 & 28 & $(1,1,1,1)[4]$ \\
2 & $\Ber^{(1)}(2,3,3,5)$ & 14 & 12 & 90 & $(4,3,3,2)[12]$ \\
3 & $\Ber^{(1)}(2,2,5,5)$ & 15 & 6 & 48 & $(2,2,1,1)[6]$ \\
4 & $\Ber^{(1)}(1,5,5,5)$ & 17 & 6 & 54 & $(3,1,1,1)[6]$ \\
5 & $\Ber^{(1)}(1,3,7,7)$ & 19 & 8 & 80 & $(4,2,1,1)[8]$ \\
6 & $\Ber^{(1)}(1,4,4,9)$ & 19 & 10 & 100 & $(5,2,2,1)[10]$ \\
7 & $\Ber^{(1)}(2,2,3,11)$ & 19 & 12 & 120 & $(4,4,3,1)[12]$ \\
8 & $\Ber^{(1)}(1,3,5,11)$ & 21 & 12 & 132 & $(6,3,2,1)[12]$ \\
9 & $\Ber^{(1)}(1,2,11,11)$ & 26 & 12 & 162 & $(6,4,1,1)[12]$ \\
10 & $\Ber^{(1)}(1,2,9,14)$ & 27 & 30 & 420 & $(15,10,3,2)[30]$ \\
11 & $\Ber^{(1)}(1,3,4,19)$ & 28 & 20 & 290 & $(10,5,4,1)[20]$ \\
12 & $\Ber^{(1)}(1,2,8,17)$ & 29 & 18 & 270 & $(9,6,2,1)[18]$ \\
13 & $\Ber^{(1)}(1,2,7,23)$ & 34 & 24 & 420 & $(12,8,3,1)[24]$ \\
14 & $\Ber^{(1)}(1,2,6,41)$ & 51 & 42 & 1092 & $(21,14,6,1)[42]$ \\
\bottomrule
\end{longtable}

\begin{longtable}{@{}r l r r r l@{}}
\caption{$m=5$, $k=4$, $p=1$: the $38$ solutions with $D\leq40$ from a total of $147$.}\label{tab:m5p1}\\
\toprule
No. & Type & $D$ & $s$ & $h$ & $(x_i)[s]$\\
\midrule
\endfirsthead
\toprule
No. & Type & $D$ & $s$ & $h$ & $(x_i)[s]$\\
\midrule
\endhead
1 & $\Ber^{(1)}(4,4,4,4,4)$ & 21 & 5 & 55 & $(1,1,1,1,1)[5]$ \\
2 & $\Ber^{(1)}(3,3,5,5,5)$ & 22 & 12 & 138 & $(3,3,2,2,2)[12]$ \\
3 & $\Ber^{(1)}(2,5,5,5,5)$ & 23 & 6 & 72 & $(2,1,1,1,1)[6]$ \\
4 & $\Ber^{(1)}(3,3,3,7,7)$ & 24 & 8 & 100 & $(2,2,2,1,1)[8]$ \\
5 & $\Ber^{(1)}(3,3,4,4,9)$ & 24 & 20 & 250 & $(5,5,4,4,2)[20]$ \\
6 & $\Ber^{(1)}(2,3,5,7,7)$ & 25 & 24 & 312 & $(8,6,4,3,3)[24]$ \\
7 & $\Ber^{(1)}(2,4,4,5,9)$ & 25 & 30 & 390 & $(10,6,6,5,3)[30]$ \\
8 & $\Ber^{(1)}(3,3,3,5,11)$ & 26 & 12 & 162 & $(3,3,3,2,1)[12]$ \\
9 & $\Ber^{(1)}(2,3,5,5,11)$ & 27 & 12 & 168 & $(4,3,2,2,1)[12]$ \\
10 & $\Ber^{(1)}(2,2,8,8,8)$ & 29 & 9 & 135 & $(3,3,1,1,1)[9]$ \\
11 & $\Ber^{(1)}(2,4,4,4,14)$ & 29 & 15 & 225 & $(5,3,3,3,1)[15]$ \\
12 & $\Ber^{(1)}(1,7,7,7,7)$ & 30 & 8 & 124 & $(4,1,1,1,1)[8]$ \\
13 & $\Ber^{(1)}(2,2,7,7,11)$ & 30 & 24 & 372 & $(8,8,3,3,2)[24]$ \\
14 & $\Ber^{(1)}(2,3,3,11,11)$ & 31 & 12 & 192 & $(4,3,3,1,1)[12]$ \\
15 & $\Ber^{(1)}(1,5,8,8,8)$ & 31 & 18 & 288 & $(9,3,2,2,2)[18]$ \\
16 & $\Ber^{(1)}(2,2,5,11,11)$ & 32 & 12 & 198 & $(4,4,2,1,1)[12]$ \\
17 & $\Ber^{(1)}(1,5,7,7,11)$ & 32 & 24 & 396 & $(12,4,3,3,2)[24]$ \\
18 & $\Ber^{(1)}(2,3,3,9,14)$ & 32 & 60 & 990 & $(20,15,15,6,4)[60]$ \\
19 & $\Ber^{(1)}(1,4,9,9,9)$ & 33 & 10 & 170 & $(5,2,1,1,1)[10]$ \\
20 & $\Ber^{(1)}(1,6,6,6,13)$ & 33 & 14 & 238 & $(7,2,2,2,1)[14]$ \\
21 & $\Ber^{(1)}(3,3,3,4,19)$ & 33 & 20 & 340 & $(5,5,5,4,1)[20]$ \\
22 & $\Ber^{(1)}(2,2,5,9,14)$ & 33 & 30 & 510 & $(10,10,5,3,2)[30]$ \\
23 & $\Ber^{(1)}(1,5,5,11,11)$ & 34 & 12 & 210 & $(6,2,2,1,1)[12]$ \\
24 & $\Ber^{(1)}(2,3,3,8,17)$ & 34 & 36 & 630 & $(12,9,9,4,2)[36]$ \\
25 & $\Ber^{(1)}(2,3,4,5,19)$ & 34 & 60 & 1050 & $(20,15,12,10,3)[60]$ \\
26 & $\Ber^{(1)}(2,2,5,8,17)$ & 35 & 18 & 324 & $(6,6,3,2,1)[18]$ \\
27 & $\Ber^{(1)}(1,5,5,9,14)$ & 35 & 30 & 540 & $(15,5,5,3,2)[30]$ \\
28 & $\Ber^{(1)}(2,2,4,14,14)$ & 37 & 15 & 285 & $(5,5,3,1,1)[15]$ \\
29 & $\Ber^{(1)}(1,5,5,8,17)$ & 37 & 18 & 342 & $(9,3,3,2,1)[18]$ \\
30 & $\Ber^{(1)}(2,2,6,6,20)$ & 37 & 21 & 399 & $(7,7,3,3,1)[21]$ \\
31 & $\Ber^{(1)}(1,3,11,11,11)$ & 38 & 12 & 234 & $(6,3,1,1,1)[12]$ \\
32 & $\Ber^{(1)}(2,3,3,7,23)$ & 39 & 24 & 480 & $(8,6,6,3,1)[24]$ \\
33 & $\Ber^{(1)}(1,4,5,14,14)$ & 39 & 30 & 600 & $(15,6,5,2,2)[30]$ \\
34 & $\Ber^{(1)}(1,4,7,7,19)$ & 39 & 40 & 800 & $(20,8,5,5,2)[40]$ \\
35 & $\Ber^{(1)}(1,5,6,6,20)$ & 39 & 42 & 840 & $(21,7,6,6,2)[42]$ \\
36 & $\Ber^{(1)}(1,3,9,11,14)$ & 39 & 60 & 1200 & $(30,15,6,5,4)[60]$ \\
37 & $\Ber^{(1)}(2,2,4,11,19)$ & 39 & 60 & 1200 & $(20,20,12,5,3)[60]$ \\
38 & $\Ber^{(1)}(2,2,5,7,23)$ & 40 & 24 & 492 & $(8,8,4,3,1)[24]$ \\
\bottomrule
\end{longtable}

\begin{longtable}{@{}r l r r r l@{}}
\caption{$m=6$, $k=5$, $p=1$: the $21$ solutions with $D\leq40$ from a total of $3462$.}\label{tab:m6p1}\\
\toprule
No. & Type & $D$ & $s$ & $h$ & $(x_i)[s]$\\
\midrule
\endfirsthead
\toprule
No. & Type & $D$ & $s$ & $h$ & $(x_i)[s]$\\
\midrule
\endhead
1 & $\Ber^{(1)}(5,5,5,5,5,5)$ & 31 & 6 & 96 & $(1,1,1,1,1,1)[6]$ \\
2 & $\Ber^{(1)}(3,5,5,5,7,7)$ & 33 & 24 & 408 & $(6,4,4,4,3,3)[24]$ \\
3 & $\Ber^{(1)}(4,4,5,5,5,9)$ & 33 & 30 & 510 & $(6,6,5,5,5,3)[30]$ \\
4 & $\Ber^{(1)}(3,3,7,7,7,7)$ & 35 & 8 & 144 & $(2,2,1,1,1,1)[8]$ \\
5 & $\Ber^{(1)}(4,4,4,4,9,9)$ & 35 & 10 & 180 & $(2,2,2,2,1,1)[10]$ \\
6 & $\Ber^{(1)}(3,5,5,5,5,11)$ & 35 & 12 & 216 & $(3,2,2,2,2,1)[12]$ \\
7 & $\Ber^{(1)}(3,4,4,7,7,9)$ & 35 & 40 & 720 & $(10,8,8,5,5,4)[40]$ \\
8 & $\Ber^{(1)}(2,5,7,7,7,7)$ & 36 & 24 & 444 & $(8,4,3,3,3,3)[24]$ \\
9 & $\Ber^{(1)}(3,3,5,8,8,8)$ & 36 & 36 & 666 & $(9,9,6,4,4,4)[36]$ \\
10 & $\Ber^{(1)}(2,5,5,8,8,8)$ & 37 & 18 & 342 & $(6,3,3,2,2,2)[18]$ \\
11 & $\Ber^{(1)}(3,3,5,7,7,11)$ & 37 & 24 & 456 & $(6,6,4,3,3,2)[24]$ \\
12 & $\Ber^{(1)}(4,4,4,5,5,14)$ & 37 & 30 & 570 & $(6,6,6,5,5,2)[30]$ \\
13 & $\Ber^{(1)}(3,4,4,5,9,11)$ & 37 & 60 & 1140 & $(15,12,12,10,6,5)[60]$ \\
14 & $\Ber^{(1)}(3,3,4,9,9,9)$ & 38 & 20 & 390 & $(5,5,4,2,2,2)[20]$ \\
15 & $\Ber^{(1)}(2,5,5,7,7,11)$ & 38 & 24 & 468 & $(8,4,4,3,3,2)[24]$ \\
16 & $\Ber^{(1)}(3,3,6,6,6,13)$ & 38 & 28 & 546 & $(7,7,4,4,4,2)[28]$ \\
17 & $\Ber^{(1)}(3,3,5,5,11,11)$ & 39 & 12 & 240 & $(3,3,2,2,1,1)[12]$ \\
18 & $\Ber^{(1)}(2,4,5,9,9,9)$ & 39 & 30 & 600 & $(10,6,5,3,3,3)[30]$ \\
19 & $\Ber^{(1)}(2,5,6,6,6,13)$ & 39 & 42 & 840 & $(14,7,6,6,6,3)[42]$ \\
20 & $\Ber^{(1)}(2,5,5,5,11,11)$ & 40 & 12 & 246 & $(4,2,2,2,1,1)[12]$ \\
21 & $\Ber^{(1)}(3,3,5,5,9,14)$ & 40 & 60 & 1230 & $(15,15,10,10,6,4)[60]$ \\
\bottomrule
\end{longtable}

\subsection{The \texorpdfstring{$p=2$}{p=2} families and the \texorpdfstring{$\tau$}{tau}-product}

At the level of denominator multisets, the operation
\begin{equation}
        \Ber_{\mathbf r}^{(p)}\,\tau\,\Ber_{\mathbf r'}^{(p')}
        =\Ber_{\mathbf r\cup\mathbf r'}^{(p+p')}
\label{eq:tau-product}
\end{equation}
corresponds to multiset union of arms.  In matrix terms the number of arms and the central weights add in such a way that the affine unit-fraction equation is preserved.  This explains why many early $p=2$ examples are composites of $p=1$ solutions.  A finer classification may therefore distinguish $\tau$-primitive weighted stars from decomposable ones.

\begin{longtable}{@{}r l r r r l@{}}
\caption{$m=4$, $k=2$, $p=2$: the classical affine $D_4^{(1)}$ star.}\label{tab:m4p2}\\
\toprule
No. & Type & $D$ & $s$ & $h$ & $(x_i)[s]$\\
\midrule
\endfirsthead
\toprule
No. & Type & $D$ & $s$ & $h$ & $(x_i)[s]$\\
\midrule
\endhead
1 & $D_4^{(1)}\simeq \Ber^{(2)}(1,1,1,1)$ & 5 & 2 & 6 & $(1,1,1,1)[2]$ \\
\bottomrule
\end{longtable}

\begin{longtable}{@{}r l r r r l@{}}
\caption{$m=5$, $k=3$, $p=2$: all $3$ unordered solutions.}\label{tab:m5p2}\\
\toprule
No. & Type & $D$ & $s$ & $h$ & $(x_i)[s]$\\
\midrule
\endfirsthead
\toprule
No. & Type & $D$ & $s$ & $h$ & $(x_i)[s]$\\
\midrule
\endhead
1 & $\Ber^{(2)}(1,1,2,2,2)$ & 9 & 6 & 30 & $(3,3,2,2,2)[6]$ \\
2 & $\Ber^{(2)}(1,1,1,3,3)$ & 10 & 4 & 22 & $(2,2,2,1,1)[4]$ \\
3 & $\Ber^{(2)}(1,1,1,2,5)$ & 11 & 6 & 36 & $(3,3,3,2,1)[6]$ \\
\bottomrule
\end{longtable}

\begin{longtable}{@{}r l r r r l@{}}
\caption{$m=6$, $k=4$, $p=2$: all $17$ unordered solutions.}\label{tab:m6p2}\\
\toprule
No. & Type & $D$ & $s$ & $h$ & $(x_i)[s]$\\
\midrule
\endfirsthead
\toprule
No. & Type & $D$ & $s$ & $h$ & $(x_i)[s]$\\
\midrule
\endhead
1 & $\Ber^{(2)}(2,2,2,2,2,2)$ & 13 & 3 & 21 & $(1,1,1,1,1,1)[3]$ \\
2 & $\Ber^{(2)}(1,2,2,2,3,3)$ & 14 & 12 & 90 & $(6,4,4,4,3,3)[12]$ \\
3 & $\Ber^{(2)}(1,1,3,3,3,3)$ & 15 & 4 & 32 & $(2,2,1,1,1,1)[4]$ \\
4 & $\Ber^{(2)}(1,2,2,2,2,5)$ & 15 & 6 & 48 & $(3,2,2,2,2,1)[6]$ \\
5 & $\Ber^{(2)}(1,1,2,3,3,5)$ & 16 & 12 & 102 & $(6,6,4,3,3,2)[12]$ \\
6 & $\Ber^{(2)}(1,1,2,2,5,5)$ & 17 & 6 & 54 & $(3,3,2,2,1,1)[6]$ \\
7 & $\Ber^{(2)}(1,1,1,5,5,5)$ & 19 & 6 & 60 & $(3,3,3,1,1,1)[6]$ \\
8 & $\Ber^{(2)}(1,1,1,3,7,7)$ & 21 & 8 & 88 & $(4,4,4,2,1,1)[8]$ \\
9 & $\Ber^{(2)}(1,1,1,4,4,9)$ & 21 & 10 & 110 & $(5,5,5,2,2,1)[10]$ \\
10 & $\Ber^{(2)}(1,1,2,2,3,11)$ & 21 & 12 & 132 & $(6,6,4,4,3,1)[12]$ \\
11 & $\Ber^{(2)}(1,1,1,3,5,11)$ & 23 & 12 & 144 & $(6,6,6,3,2,1)[12]$ \\
12 & $\Ber^{(2)}(1,1,1,2,11,11)$ & 28 & 12 & 174 & $(6,6,6,4,1,1)[12]$ \\
13 & $\Ber^{(2)}(1,1,1,2,9,14)$ & 29 & 30 & 450 & $(15,15,15,10,3,2)[30]$ \\
14 & $\Ber^{(2)}(1,1,1,3,4,19)$ & 30 & 20 & 310 & $(10,10,10,5,4,1)[20]$ \\
15 & $\Ber^{(2)}(1,1,1,2,8,17)$ & 31 & 18 & 288 & $(9,9,9,6,2,1)[18]$ \\
16 & $\Ber^{(2)}(1,1,1,2,7,23)$ & 36 & 24 & 444 & $(12,12,12,8,3,1)[24]$ \\
17 & $\Ber^{(2)}(1,1,1,2,6,41)$ & 53 & 42 & 1134 & $(21,21,21,14,6,1)[42]$ \\
\bottomrule
\end{longtable}
\section{Graphs}

\Cref{fig:general-graphs} shows the generic simply laced star-shaped graph and the symbolic fusion rule.  \Cref{fig:m4-graphs} displays the fourteen $m=4$, $p=1$ cases.

\begin{figure}[htbp]
\centering
\begin{tabular}{c}
\includegraphics[width=0.45\linewidth]{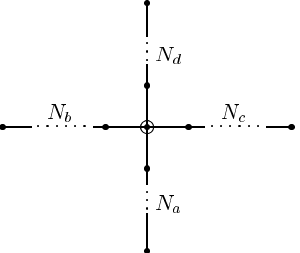}\\[0.8cm]
\includegraphics[width=0.70\linewidth]{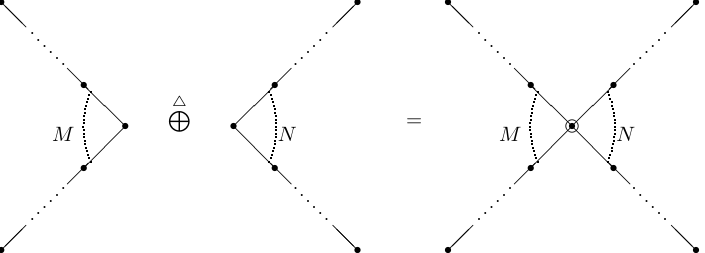}
\end{tabular}
\caption{Top: generic simply laced star-shaped weighted star graph. Bottom: symbolic form of the $\tau$-product on arm multisets.}
\label{fig:general-graphs}
\end{figure}

\begin{figure}[p]
\centering
\begin{tabular}{cccc}
\includegraphics[width=.21\linewidth]{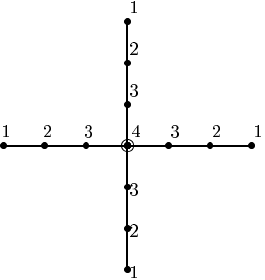} &
\includegraphics[width=.21\linewidth]{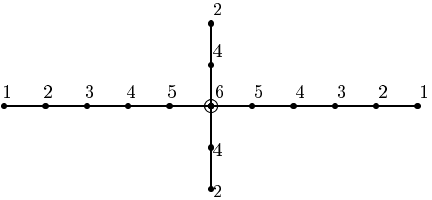} &
\includegraphics[width=.21\linewidth]{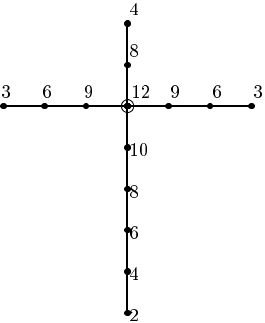} &
\includegraphics[width=.21\linewidth]{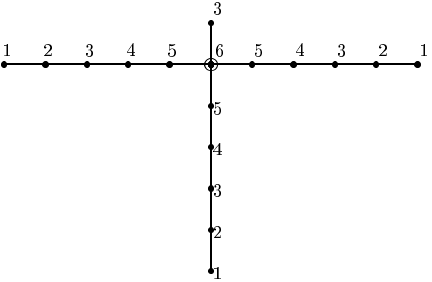} \\
\includegraphics[width=.21\linewidth]{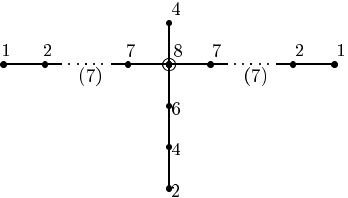} &
\includegraphics[width=.21\linewidth]{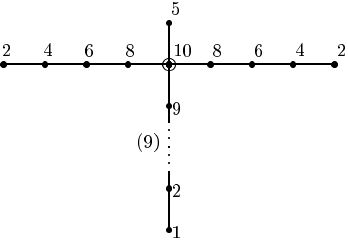} &
\includegraphics[width=.21\linewidth]{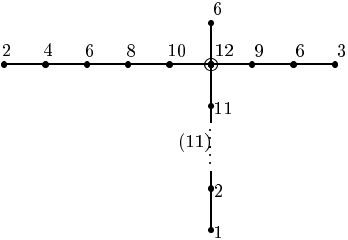} &
\includegraphics[width=.21\linewidth]{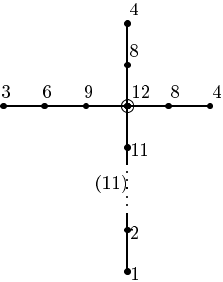} \\
\includegraphics[width=.21\linewidth]{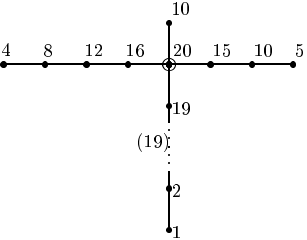} &
\includegraphics[width=.21\linewidth]{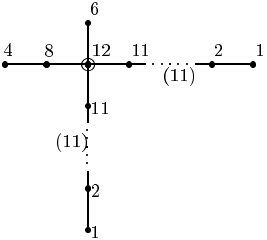} &
\includegraphics[width=.21\linewidth]{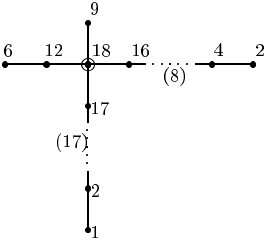} &
\includegraphics[width=.21\linewidth]{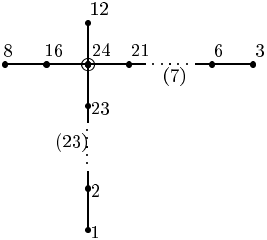}\\
\includegraphics[width=.21\linewidth]{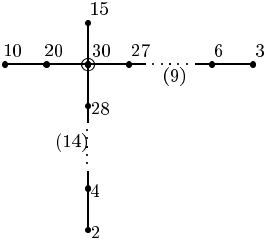} &
\includegraphics[width=.21\linewidth]{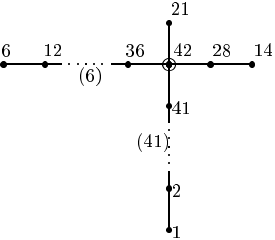} & &
\end{tabular}
\caption{The fourteen affine weighted star graphs for $m=4$, $k=3$, $p=1$.}
\label{fig:m4-graphs}
\end{figure}

\clearpage

\section{Conclusions}

This work has recast the Berger construction as a classification problem for integral symmetric matrices and weighted star graphs.  In this form the main object is not a conjectural algebra but the explicit $Z$-matrix $B(k;r_1,\ldots,r_m)$ associated with a star tree whose arms are type-$A$ paths and whose central vertex has weight $k$.

The main structural result is that the determinant, inertia and affine condition are all governed by the single Schur-complement scalar
\[
        k-m+\sum_{i=1}^m\frac{1}{r_i+1}.
\]
Thus the positive-definite, affine and one-negative-eigenvalue regimes are completely separated.  In the affine case the classification is exactly the Egyptian-fraction equation
\[
        \sum_{i=1}^m\frac{1}{r_i+1}=m-k,
\]
and the positive kernel vector is given explicitly by the least common multiple of the denominators.  

The Coxeter-label formula and the divisor partition \eqref{eq:divisor-partition} provide compact invariants for comparing the resulting weighted stars.

This matrix/graph perspective also clarifies the relation to known classifications.  When all diagonal entries are $2$, the spectral-radius classification of simply laced Dynkin and affine Dynkin diagrams is recovered.  When the central diagonal is allowed to vary, the same star-shaped support gives finite families for fixed $(m,k)$, enumerated by unit fractions.  The cases $D_4^{(1)}$, $E_6^{(1)}$, $E_7^{(1)}$ and $E_8^{(1)}$ are therefore not isolated phenomena but the first classical members of a broader weighted-star enumeration.

Several natural problems remain.  One can study $\tau$-primitive affine stars, extend the classification from simply laced stars to more general weighted trees, compare the resulting matrices with cyclotomic and symmetrizable integer-matrix classifications, and connect the unit-fraction constraints with quiver, canonical-algebra and weighted-projective-line classifications.  A separate direction is to investigate which of these graphs arise from singularity theory, McKay-type correspondences, weighted homogeneous singularities or toric Calabi--Yau constructions \citep{McKay1980,EbelingTakahashi2011}.  Those questions are deliberately separated from the present contribution: the results proved here are the linear-algebraic and combinatorial core needed before any stronger geometric or representation-theoretic interpretation can be justified.

\section*{Acknowledgements}

We acknowledge support from  the Spanish MEC-CYCIT and Fundacion Seneca (CARM Murcia) funding agencies and, at some stages of this work, hospitality from the CERN Theory Division and the Department of Theoretical Physics C-XI of the Universidad Aut\'onoma de Madrid.

\end{document}